\documentclass[a4paper,12pt,twoside]{amsart}
\usepackage{amsthm,amssymb,amsmath,amscd,url}
\newtheorem{thm}{Theorem}[section]

\newtheorem{lem}[thm]{Lemma}

\newcommand{\Z}{\mathbb{Z}}

\newcommand{\R}{\mathbb{R}}
\newcommand{\C}{\mathbb{C}}

\newcommand{\Q}{\mathbb{Q}}

\def\fS{{\mathfrak S}}

\def\PPhi{{\Phi}}

\def\cB{{\mathcal B}}

\def\cJ{{\mathcal J}}

\def\SL{{\rm SL}}

\def\GL{{\rm GL}}

\def\End{{\rm End}}

\def\Ind{{\rm Ind}}

\def\Irr{{\mathbf {Irr}}}

\def\Prim{{\rm Prim}}

\def\cG{{\mathcal G}}
\def\cZ{{\mathcal Z}}

\def\cT{{\mathcal T}}

\def\fa{{\mathfrak a}}
\def\fb{{\mathfrak b}}
\def\fc{{\mathfrak c}}
\def\fd{{\mathfrak d}}

\def\fM{{\mathfrak M}}

\def\ft{{\mathfrak t}}

\def\fB{{\mathfrak B}}

\def\fI{{\mathfrak I}}

\def\fK{{\mathfrak K}}
\def\fN{{\mathfrak N}}
\def\fL{{\mathfrak L}}

\def\fA{{\mathfrak A}}
\def\fB{{\mathfrak B}}

\def\Ad{{\rm Ad}}

\def\q{{/\!/}}

\def\PPhi_F{{\rm Frob}}

\def\LA{{\overline{A}}}

\subjclass[2010]{22E50, 20G05}
\begin{document}
\title[$K$-theory and the connection index]{$K$-theory and the connection index}
\author{Tayyab Kamran and Roger Plymen}
\address{T. Kamran: Centre for Advanced Mathematics and
Physics, National University of Sciences and
Technology, H-12, Islamabad, Pakistan}
\email{tkamran@camp.nust.edu.pk}
\address{R.~Plymen: School of Mathematics, Manchester University, Manchester M13 9PL, England
\emph{and} School of Mathematics, Southampton University, SO17 1BJ, England}
\email{plymen@manchester.ac.uk, r.j.plymen@soton.ac.uk}
\date{\today}
\maketitle
\thanks{}

\begin{abstract}  
Let $\cG$ denote a split simply connected almost simple $p$-adic group. The classical example is the special linear group $\SL(n)$.   We study the $K$-theory of the unramified unitary principal series of $\cG$ and prove that the rank of  $K_0$ is the connection index $f(\cG)$.  We relate this result to a recent refinement of the Baum-Connes conjecture, and show explicitly how generators of $K_0$ contribute to the  $K$-theory of the Iwahori $C^*$-algebra $\fI(\cG)$.  
\end{abstract}

\section{Introduction}     Let $\cG$ denote a split simply connected almost simple $p$-adic
group. The classical example is the special linear group $\SL(n)$.   We study the $K$-theory of the unramified unitary principal series of  $\cG$ and prove
that the rank of $K_0$ is the connection index $f(\cG)$.

In the spirit of noncommutative geometry, we construct a noncommutative $C^*$-algebra, the  \emph{spherical} $C^*$-algebra $\mathfrak{S}(\cG)$. The primitive ideal spectrum $\Prim \, \fS(\cG)$ of $\fS(\cG)$ can be identified with the irreducible representations in the unramified unitary principal series of $\cG$.   The $C^*$-algebra 
$\fS(\cG)$ is a direct summand of the Iwahori $C^*$-algebra $\fI(\cG)$, and so contributes to the $K$-theory of $\fI(\cG)$. 

We  relate this result with the recent conjecture in \cite[\S7]{ABP3}: this is a version, adapted to the $K$-theory of $C^*$-algebras, 
of the geometric conjecture developed in  \cite{ABP1},\cite{ABP2},\cite{ABP3},\cite{ABP4}.  

Quite specifically, let $F$ be a local nonarchimedean field of
characteristic $0$, let $\cG$ be the group of $F$-rational points
in a split, almost simple, simply connected, semisimple linear
algebraic group defined over $F$, for example $\SL(n)$. Let $\cT$
denote a maximal split torus of $\cG$. Let $ \cG^{\vee},
\,\cT^{\vee}$ denote the Langlands dual groups, and let $G,T$
denote maximal compact subgroups:
\[
G \subset \, \cG^{\vee}, \quad T \subset \,\cT^{\vee}.\] Then $G$
is a compact connected Lie group, of adjoint type,  with maximal torus $T$.

Let $\pi_1(G)$ denote the fundamental group of $G$.   The order of $\pi_1(G)$ is the \emph{connection index}.   The connection index is a numerical invariant attached to 
$\cG$ denoted $f(\cG)$.  The notation $f$ is due to Bourbaki \cite[VI, p.240]{B}. 

The fundamental groups  are listed in \cite[Plates I--X,
p.265--292]{B}.  They are 
\[
\Z/(n+1)\Z,\quad \Z/2\Z,\quad \Z/2\Z,\quad \Z/4\Z\; (n\;
odd),\quad \Z/2\Z \times \Z/2\Z \;(n\; even)\] \[\Z/3\Z,\quad
\Z/2\Z, \quad 0,\quad 0,\quad 0
\]
for $A_n, B_n, C_n, D_n, E_6, E_7, E_8, F_4, G_2$, respectively.

The primitive ideal spectrum $\Prim \, \fS(\cG)$ can be identified with the unramified unitary principal series of $\cG$.  In this identification, each element in $\Prim \, \fS(\cG)$ is an irreducible constituent of a representation induced, by parabolic induction, from an unramified unitary character of $\cT$.  

When we  compare the $R$-group computations in \cite{K} with the
above list, we see that \emph{the maximal $R$-groups are
isomorphic to the fundamental groups}.   In this article, we explain this fact, via a geometrical approach.

The  $L$-packet in the unramified unitary principal series of $\cG$ with the maximal number of irreducible constituents has $f(\cG)$ constituents. 

\begin{thm}\label{main} Let $\fS(\cG)$ denote the spherical $C^*$-algebra  and let $f(\cG)$
be the connection index of $\cG$.  Then $K_0 \,\fS(\cG)$ is a free abelian group on $f(\cG)$ generators, and $K_1\, \fS(\cG) = 0$.    
 \end{thm}
 
 In Theorem (\ref{class}) we show explicitly how generators of $K_0 \, \fS(\cG)$ contribute to the $K$-theory of the Iwahori
 $C^*$-algebra $\fI(\cG)$. 
 
 Intuitively, we can observe a deformation retraction of the spherical $C^*$-algebra $\fS(\cG)$ onto the $L$-packet with $f(\cG)$ constituents. This $L$-packet is tempered.   When $\cG = \SL(n,F)$, this $L$-packet is elliptic; see \cite[Theorem 3.4]{G}.      Elliptic representations share with the discrete series the property that their Harish-Chandra characters are not identically zero on the regular elliptic set; and $K_0$ sees these elliptic representations as if they were $f(\cG)$ isolated points in the discrete series.

\section{The alcoves in the Lie algebra of $T$}

The
proof depends on the distinction between the affine Weyl group $W_a$ and
the extended affine Weyl group $W'_a$.   The quotient $W'_a/W_a$ is a finite abelian group which dominates the discussion.

Our reference at this point is \cite[IX, p.309--327]{B}.  Let $\ft$ denote the Lie algebra
of $T$, and let $\exp : \ft \to T$ denote the exponential map. The
kernel of $\exp$ is denoted $\Gamma(T)$.    The inclusion $\iota : T \to G$ induces the homomorphism $\pi_1(\iota): \pi_1(T) \to \pi_1(G)$.    
Now $f(G,T)$ will denote the composite of the canonical isomorphism
from $\Gamma(T)$ to $\pi_1(T)$ and the homomorphism  $\pi_1(\iota)$:
\[
f(G,T) :  \Gamma(T) \simeq  \pi_1(T) \to \pi_1(G).
\]
Denote by $N(G,T)$ the kernel of $f(G,T)$.   We have a short exact sequence 
\[
0 \to N(G,T) \to \Gamma(T) \to \pi_1(G).
\]

Denote by $N_G(T)$ the normalizer of $T$ in $G$.   Let $W$ denote the Weyl group $N_G(T)/T$.     The affine Weyl group is $W_a = N(G,T) \rtimes W$ and the extended
affine Weyl group is $W'_a = \Gamma(T) \rtimes W$; the subgroup
$W_a$ of $W'_a$ is normal. 

If $\ft - \ft_r$ denotes the union of the singular hyperplanes in $\ft$, then the \emph{alcoves} of $\ft$ are the connected component of $\ft_r$. 

The group $W_a$ operates
simply-transitively on the set of alcoves. Let
 $A$ be an an alcove. Then 
$\overline{A}$ is a fundamental domain for the operation of $W_a$
on $\ft$. 

Let $H_A$ be the stabilizer of
$A$ in $W'_a$.   Then $H_A$ is a finite abelian group which can be identified naturally with $\pi_1(G)$, see \cite[IX, p.326]{B}.    
The extended affine Weyl group  $W'_a$ 
is the semi-direct product
\[
W'_a = W_a \rtimes H_A.\]

View  $\ft$ as an additive group, and form the Euclidean group $\ft
\rtimes O(\ft)$. We have $W'_a \subset \ft \rtimes O(\ft)$ and so
$W'_a$ acts as affine transformations of $\ft$.   Now $H_A$ leaves $\overline{A}$ invariant, so $H_A$ acts as affine transformations of $\overline{A}$. 
Let $v_0, v_1, \ldots, v_n$ be the vertices of the simplex $\overline{A}$.   We will use  barycentric coordinates, so that 
\[
x = \sum_{i = 0}^n t_iv_i
\]
with $x \in \overline{A}$.  The barycentre $x_0$ has coordinates $t_j = 1,\; 0\leq j \leq n$ and so is $H_A$-fixed.   
Then $\overline{A}$ is equivariantly contractible to
$x_0$: \begin{align}\label{con} r_t(x): = tx_0 + (1-t)x
\end{align}with $0 \leq t
\leq 1$. This is an affine $H_A$-equivariant retract from
$\overline{A}$ to $x_0$.\smallskip

\begin{lem}\label{H}  Let $t_0 = \exp x_0$. There is a canonical isomorphism 
\[
W(t_0) \simeq H_A.
\]
\end{lem}

\begin{proof}  
Let $w \in W$.   We have $w\cdot t_0 = t_0$ if and only if there exists $\gamma \in \Gamma(T)$, uniquely determined by $w$, for which $\gamma(w(x_0)) = x_0$.  But $\gamma w$ will fix $x_0$ if and only if $\gamma w$ stabilizes $A$, i.e. $\gamma w \in H_A$.  This determines the isomorphism
\[
W(t_0) \simeq H_A, \quad w \mapsto \gamma w
\]
\end{proof}

In the special case of $\SL(3)$, the vector space $\ft$ is the Euclidean
plane $\R^2$.  The singular hyperplanes 
tessellate $\R^2$ into equilateral triangles. The interior of each
equilateral triangle is an alcove.  Barycentric subdivision refines this tessellation into
isosceles triangles.  The extended affine Weyl group $W'_a$ acts simply transitively on the set of these
isosceles triangles, but the closure $\overline{\Delta}$ of one such triangle 
is not a fundamental domain for the action of $W'_a$.  The corresponding quotient space is
\cite[IX.\S5.2]{B}:
\[
\ft/W'_a \simeq \overline{A}/H_A.\] The abelian group $H_A$ is the
cyclic group $\Z/3\Z$ which acts on $\overline{A}$ by rotation about the barycentre of $\overline{A}$ 
through $2\pi/3$.  

\section{The  spherical $C^*$-algebra}  We will focus on the $C^{*}$-summand $\fS(\cG)$ in the reduced
$C^{*}$-algebra $C^{*}_{r}(\cG)$ which corresponds to the unramified
unitary principal series, see \cite{P}. The algebra $C^{*}_{r}(\cG)$
is defined as follows. We choose a left-invariant Haar measure on
$\cG$, and form a Hilbert space $L^{2}(\cG)$. The left regular
representation $\lambda$ of $L^{1}(\cG)$ on $L^{2}(\cG)$ is given by
\[
(\lambda(f))(h)=f*h
\]
where $f \in L^{1}(\cG)$, $h \in L^{2}(\cG)$ and $*$ denotes the
convolution. The $C^{*}$-algebra generated by the image of
$\lambda$ is the reduced $C^{*}$-algebra $C^{*}_{r}(\cG)$. Let $\fK$
denote the $C^{*}$-algebra of compact operators on the standard
Hilbert space $H$.

As in \S1, the Langlands dual
of $\cG$ is the complex reductive group $\cG^{\vee}$ with maximal
torus $\cT^{\vee}$. Let $G$ be a maximal compact subgroup of
$\cG^{\vee}$, let $T$ be the maximal compact subgroup of
$\mathcal{T}^{\vee}$.

The group $\Psi(\mathcal{T})$ of unramified unitary characters of
$\mathcal{T}$ is isomorphic to $T$.  The \emph{spherical} $C^*$-algebra
is given by the fixed point algebra
\begin{align*}\label{eq:1}
\fS(\cG):&= C(T,\fK)^{W} \nonumber \\
&= \{f \in C(T,\fK):f(wt)=\fc(w:t)\cdot f(t), w\in W\}
\end{align*}
where $\mathfrak{a}(w:t)$ are normalized intertwining operators,
and \[\fc(w:t) = \Ad\, \fa(w:t)\]  as in \cite{P}. Then $\fa : W
\to C(T,U(H))$ is a $1$-cocycle:
\[
\fa(w_2w_1:t) = \fa(w_2:w_1t)\fa(w_1:t).\]

\begin{thm}  The group $K_0(\fS(\cG))$ is free abelian on $f(\mathcal{G})$ generators, and $K_1 =0$.
\end{thm}

\begin{proof} We have the exponential map $\exp: \ft \to T$. We lift $f$ from
$T$ to a periodic function $F$ on $\ft$, and lift $\fa$ from a
$1$-cocycle $\fa: W \to C(T,U(H))$ to a $1$-cocycle  $\fb : W'_a
\to C(\ft, U(H))$:
\[ F(x): = f(\exp x), \quad
\fb(w':x): = \fa(w: \exp x)\] with $w' = (\gamma, w)$.  The
semidirect product rule is
\[w'_1w'_2 = (\gamma_1,w_1)(\gamma_2,w_2) = (\gamma_1w_1(\gamma_2),
w_1w_2).\] Note that $\fb$ is still a $1$-cocycle:
\begin{align*} 
\fb(w'_2w'_1:x) =& \fa(w_2w_1: \exp x) \\
=& \fa(w_2:w_1(\exp x)\fa(w_1: \exp x)
\\
=& \fa(w_2: \exp w_1x)\fb(w'_1:x)
\\
=& \fa(w_2: \exp \gamma w_1x)\fb(w'_1:x)
\\
=& \fb(w'_2: w'_1 x)\fb(w'_1:x) \nonumber\\
\end{align*}
Now we define
\begin{align*}
\fd(w:x): = Ad \, \fb(w:x), \quad \quad w \in W'_a
\end{align*}
 The fixed algebra $\fS(\cG)$ is as follows:
\[
\{F \in C(\ft,\fK):F(wx)=\fd(w:x) \cdot F(x), w\in W'_a,\;x \in
\ft,\, F \;\textrm{periodic}\}
\]

Now $F$ is determined by its restriction to $\overline{A}$. Upon
restriction, we obtain
\[\fS(\cG) \simeq \{f \in
C(\overline{A},\fK):f(wx)=\fd(w:x)\cdot f(x), w\in H_A,\;x \in
\overline{A}\}
\]

We will
write $H = H_A$.  Let \[ \fA: = C(\overline{A},\fK), \quad \quad
\fB: = C(\overline{A},\fL(V))\]so that $\fS(\cG) = \fA^H$.  

   We apply \cite[Theorem 2.13]{PL}.  We have to verify that all three conditions of this theorem are met.
Let $R(t_0)$ denote the $R$-group attached to $t_0$.  The calculations of Keys\cite{K} show that
$R(t_0) = W(t_0)$ and so condition (**) in \cite{PL} is met.   We note that $\{\{\fb(w:x) : w \in H_A\}  = \{\fa(w:t) : w \in W(t_0)\}$ and that $w \mapsto \fa(w: t_0)$ is a unitary representation of $H_A$.  For this representation we have
\begin{align*}
\fa(w : t_0) = \rho_1(w)P_1 + \cdots  + \rho_f(w)f P_f
\end{align*}
where $\rho_1, \ldots, \rho_f$ are the characters of $H_A$, and $P_1,   \ldots, P_f$ are the orthogonal projections onto the $f$ irreducible subspaces 
$V_1, \ldots, V_f$ of the induced representation $\Ind_{\cB}^{\cG}\, t_0$.    The representation $\fa( - : t_0)$ is quasi-equivalent to the regular representation of $H_A$, i.e. it contains the same characters (with different multiplicity), and so condition (***) is met.  Let $W(t)$ denote the isotropy subgroup of $t \in T$. We have to compare the two following representations of representations of $W(t)$:
\begin{align}\label{reps}
w \mapsto \fa(w : t_0) _{| W(t)} , \quad \quad w \mapsto \fa(w : t)
\end{align}
These two representations are quasi-equivalent, again by inspecting the results of Keys \cite {K}.   Since $f$ can be non-prime only in cases $A_n$ and $D_n$, these are the only two cases to be checked.   In each subspace $V_1 , \ldots, V_f $ choose an increasing sequence $e_1^n, \ldots, e_f^n$ of projections tending strongly to the identity operator on that subspace, and set $e_n = e_1^n + \cdots + e_f^n$.  The compressions, with respect to the projections $(e_n)$, of the two representations  in (\ref{reps}) remain quasi-equivalent for each $n$:  they each contain every character of $W(t)$.   The  condition (*) is now met.

The three conditions of  \cite[Theorem 2.13]{PL} are met.  This yields a strong Morita equivalence
 \[
 \fA^H \simeq_{Morita} C(\overline{A}) \rtimes H
 \]
 We recall that
 \[
 C(\LA) \rtimes H \simeq C(\LA, \End \, \C H)^H
 \]
 where $H$ acts via the regular representation on its group algebra $\C H$.   
Define
 \[\fM: = C(\LA, \End \, \C H)^H, \quad \quad \fN : = (\End \, \C H)^H
 \]
 and define the homomorphisms 
 \[ 
 f : \fM \to \fN, \quad \sigma \mapsto \sigma(x_0)
 \]
 \[
 g : \fN \mapsto \fM, \quad Y \mapsto  Y(-).
 \]
 Then $(f \circ g) = id_{\fN}$ and $g \circ f$ sends the map $\sigma$ to the constant map $\sigma(x_0)(-)$.  With $r_t$ as in Eqn.(\ref{con}), we define
 \[
 F_t (\sigma) = \sigma \circ r_t
 \]
 then $F_1 = g \circ f$ and $F_0 = id_{\fM}$. So $g \circ f \sim_h id_{\fM}$.  Therefore,  $\fM$ and $\fN$ are homotopy equivalent $C^*$-algebras, and have the same $K$-theory.
The fixed $C^*$-algebra $\fM$ is homotopy equivalent to its fibre $\fN$ over the fixed point $x_0$. 
  
  \medskip
  
 We therefore have
 \begin{align*}
 \fS(\cG)  = \fA^H \sim_{Morita} \fM \sim_h  \fN \simeq \C^f
 \end{align*}
 The $C^*$-algebra  functors $K_0$ and $K_1$ are invariants  of strong Morita equivalence and of homotopy type, whence
\begin{align*}
K_0 \,\fS(\cG) & = \Z^f \\
 K_1 \,\fS(\cG) & = 0
\end{align*}
where $f = f(\cG)$. 
\end{proof}

Let $\cJ$ be a good maximal compact subgroup of $\cG$. Choose  left-invariant Haar measure $\mu$ 
on $\cG$. Let
\begin{align*}
e_{\cJ}(x) & = \mu(\cJ)^{-1} \; \textrm{if} \; x \in \cJ\\
& = 0 \quad \quad \quad  \textrm{if} \; x \notin \cJ
\end{align*}
Then $e_{\cJ}$ is an idempotent in the reduced $C^*$-algebra 
$C^*_r(\cG)$. 
Let $\cJ_1, \ldots, \cJ_f$ be an enumeration of the good maximal compact subgroups of $\cG$, one from each conjugacy class in $\cG$.   Let $e_{\cJ_1}, \ldots, e_{\cJ_f}$ be the corresponding idempotents.  Their Fourier transforms are rank-one projections in $\fK(E_1), \ldots, \fK(E_f)$, and serve as generators for $K_0$.

\section{The $L$-packet with $f(\cG)$ constituents}

The $H_A$-fixed point  $t_0 \in T$ determines an $L$-parameter:
\[
\phi : W_F \times \SL(2,\C) \to G, \quad \quad (y \Phi_F^n,Y) \mapsto t_0^n
\]
with $W_F$ the Weil group of $F$, $I_F$ the inertia subgroup of $W_F$, $y \in I_F$ , $\Phi_F$ a geometric Frobenius  in  $W_F$, $Y \in \SL(2,\C)$.

The $L$-parameter $\phi$ determines an $L$-packet $\Pi(t_0)$. The constituents of this $L$-packet have Kazhdan-Lusztig parameters
\[
\{(t_0,1,\rho) : \rho \in \Irr \, \pi_0 \, \mathcal{Z}_G(t_0)\}
\]
 where $\mathcal{Z}_G(t_0)$ is the centralizer of $t_0$ in $G$.  We have 
 \[
 \mathcal{Z}_G(t_0) = \cT^{\vee} \cdot W(t_0)
 \]
 and so
 \[
 \pi_0 \, \cZ_G(t_0) = W(t_0).
 \]
By Lemma (\ref{H}), we have $W(t_0) \simeq H_A$.  It follows that the third Kazhdan-Lusztig parameter $\rho$ is a character of the abelian group $H_A$.  All such characters are allowed.   Since $H_A$ can be naturally identified with $\pi_1(G)$, the order of $H_A$ is the connection index $f(\cG)$.   As a consequence, the number of irreducible constituents in the $L$-packet $\Pi(t_0)$ is  equal to the connection index $f(\cG)$. 

\section{The Baum-Connes correspondence: a refinement}   
  Let $\cG$ be a reductive $p$-adic group.   The Baum-Connes correspondence is a definite isomorphism of abelian groups
\begin{align*}
K_j^{top}(\cG) \simeq K_j \, C^*_r(\cG)
\end{align*}
with $j = 0,1$, see \cite{VL}.  The left-hand-side, defined in terms of $K$-cycles, has never been directly computed for a noncommutative reductive $p$-adic group.   A result of Higson-Nistor 
\cite{H} and Schneider  \cite{PS} allows us to replace the left-hand-side with the chamber homology groups.   Chamber homology has been directly computed for only two noncommutative $p$-adic groups:  $\SL(2)$, see \cite{BHP} and $\GL(3)$, see \cite{AHP}.   In the case of $\GL(3)$, one can be sure that representative cycles in all the homology groups have been constructed only by checking with the right-hand-side.    
In other words, one always has to have an independent computation of the right-hand-side.  
\medskip

 We now reflect on the conjecture in \S7 of \cite{ABP3}.    This is the geometric conjecture  developed in \cite{ABP1},\cite{ABP2},\cite{ABP3},\cite{ABP4} adapted to the $K$-theory of $C^*$-algebras.      We will focus on the Iwahori $C^*$-algebra  $\fI(\cG) \subset C^*_r(\cG)$.    The primitive ideal spectrum of $\fI(\cG)$ can be identified with the irreducible tempered representations of $\cG$ which admit nonzero Iwahori-fixed vectors. 
 
 \emph{In this special case}, the conjecture in \cite[\S7]{ABP3} asserts that 
\begin{align}\label{KKK}
K_j(\fI(\cG)) = K^j_{W}(T) 
\end{align}
with $j = 0,1$.   Here $K^j_W(T)$	 is the classical topological equivariant K-theory \cite[\S2.3]{AHS}
 for the Weyl group $W$ acting on the compact torus $T$. 
 
 We quote a recent theorem of Solleveld: 
\begin{align*}
K_j(\fI(\cG))\otimes_{\Z}\Q \simeq K_j( C^*(W'_a))\otimes_{\Z}\Q
\end{align*}
with $j = 0,1$.     This theorem is a special case of \cite[Theorem 5.1.4]{S2} when $\Gamma = 1$.
We have 
\begin{align*}
C^*(W'_a) & =   C^*(\Gamma(T) \rtimes W) \\
& \simeq C(T) \rtimes W
\end{align*}
by a standard Fourier transform.  By the Green-Julg theorem \cite[Theorem 11.7.1]{Black}, we have
\begin{align*}
K_j(C(T) \rtimes W) \simeq K_W^j(T).
\end{align*}
Therefore we have
\begin{align}\label{Sol}
K_j(\fI(\cG))\otimes_{\Z}\Q \simeq K^j_W(T)\otimes_{\Z}\Q.
\end{align}
so that Solleveld establishes Eqn.(\ref{KKK})  \emph{modulo torsion}.    
\smallskip 

 Applying the equivariant Chern character for discrete groups \cite{BC} gives a
map
\begin{align}\label{ch}
ch_W : K^j_W(T) \to \bigoplus_{l}  H^{j+2l}(T\q W; \C) 
\end{align}
 which becomes an isomorphism when $K^j_W(T)$ is tensored with $\C$.
Hence the geometric conjecture at the level of $C^*$-algebra $K$-theory gives a much finer and more precise formula for
$K_*C^*_r(\cG)$ than Baum-Connes alone provides; see \cite{ABP3}.   

In the formula (\ref{ch}) for the equivariant Chern character, $T\q W$ denotes the extended quotient of $T$ by $W$. This is easily computed, as is clear from its definition, which we recall briefly.     The quotient $T/W$ is obtained by collapsing each orbit to a point and is again a compact Hausdorff space.  
  For $t \in T$, $W(t)$ denotes the stabilizer group of $t$:
\[
W(t): = \{w \in W : wt = t\}.
\]
Denote by $c(W(t))$ the set of conjugacy classes of $W(t)$.  The extended quotient, denoted $T\q W$, is constructed by replacing each orbit with $c(W(t))$ where $t$ can be any point in the orbit. This construction is done as follows.   First, set
\[
\widetilde{T}: = \{(w,t) \in W \times T : wt = t\}.
\]
Then $\widetilde{T} \subset W \times T$.  The group $W$  acts on $\widetilde{T}$:
\[
W \times \widetilde{T} \to \widetilde{T}, \quad \quad \alpha(w,t) = (\alpha w \alpha^{-1}, \alpha t)
\]
with $(w,t) \in \widetilde{T}, \alpha \in W$.  The extended quotient is defined by
\[
T\q W: = \widetilde{T}/W.
\]
Hence the extended quotient $T\q W$ is the ordinary quotient for the action of $W$ on $\widetilde{T}$. 

\section{Equivariant line bundles over $T$}  It follows from the $C^*$-Plancherel theorem \cite{P}  that the $C^*$-ideal $\fS(\cG)$ is a direct summand of the $C^*$-algebra $\fI(\cG)$.  We therefore have
\begin{align}\label{K}
K_0\, \fS(\cG) \hookrightarrow K_0 \, \fI(\cG).
\end{align}

When we combine (\ref{K}) with (\ref{Sol}), we obtain
\begin{align}\label{KK}
K_0 \, \fS(\cG) \hookrightarrow K^0_W(T) \otimes \C
\end{align}

We will now make the map (\ref{KK}) explicit.  We recall that the group $\psi(\cT)$ of unramified unitary characters of the maximal torus $\cT$ is isomorphic to the compact torus $T$.    The Hilbert space of the induced representation
\[
\pi_t : = \Ind_{\cB}^{\cG} \, t
\]
 will be denoted $E_t$ with $ t \in T$.   We recall from \S 4  that
\[
\Ind_{\cB}^{\cG} \, t_0 = \Pi(t_0)
\]
is an $L$-packet with $f$ irreducible constituents.

The collection $\{E_t : t \in T\}$ forms a $W$-equivariant bundle of Hilbert spaces via the action of the intertwining operators:
\[
\fa(w:t) : E_t \to E_{wt}
\]
We will construct $1$-dimensional sub-bundles of fixed vectors as follows.
 Let $\cJ$ be a good maximal compact subgroup of $\cG$. 
Let $e_{\cJ}$ be the corresponding idempotent in $C^*_r(\cG)$.   Let $\pi$ be a unitary representation of $\cG$ on the Hilbert space $E$.  We define the linear operator
\[
\pi(\phi): = \int_{\cG} \phi(g)\pi(g) dg
\]
whenever $\phi$ is a suitable test-function.   When $\phi = e_{\cJ}$, the linear operator $\pi(e_{\cJ})$ is a projection onto the subspace $E^{\cJ}$ of $\cJ$-fixed vectors.   As we vary $\pi$ in the primitive ideal spectrum $\Prim \, \fS(\cG)$, we obtain 
 a continuous field of rank $1$ projections over the compact torus $T$.   This idempotent projects onto the complex hermitian line bundle $E^{\cJ}$ over $T$ of $\cJ$-fixed vectors. The fibre of $E^{\cJ}$ at the point $t$ is given by
 \[ 
E^{\cJ}_t =  \pi_t (e_{\cJ}) E_t
 \]
 
 This creates a $W$-equivariant line bundle over $T$. As we vary $\cJ$ among $\cJ_1, \ldots, \cJ_f$, we obtain 
 $f$ distinct $W$-equivariant line bundles over $T$.  To show that they are distinct, we proceed as follows.  
 
 Let $\cJ = \cJ_k$. Then, for each $1 \leq k \leq f$,  the fibre of $E^{\cJ}$ at the point $t_0 \in T$ is  a $1$-dimensional subspace  of the irreducible subspace $V_k$ of the $L$-packet $\Pi(t_0)$, by a result of Keys on class$-1$ representations  [Theorem, \S 4]\cite{K}. 
The irreducible subspaces $V_1, \ldots, V_f$ are \emph{inequivalent} representations of $\cG$, and so the $W$-equivariant line bundles $E^{\cJ}$ with $1 \leq k \leq f$ are distinct.  This establishes the next result.
 
 \begin{thm}\label{class}The injective homomorphism (\ref{KK}) is given explicitly by
 \begin{align*}\label{KKK}
K_0\, \fS(\cG) \hookrightarrow K^0_W(T) \otimes \C, \quad \quad e_{\cJ} \mapsto E^{\cJ}
\end{align*}
with $\cJ = \cJ_1, \ldots, \cJ_f$. The image of this map is a free abelian group of rank  $f$.
\end{thm}

For the exceptional groups $\cG = E_8,\; F_4$ or $G_2$, the connection index $f = 1$, and so the map (\ref{K}) is especially simple. There is no $L$-packet. The $K$-theory of the unramified unitary principal series of $\cG$ is that of a point.  

\bigskip

\emph{Acknowledgement}. We would like to thank Anne-Marie Aubert for her careful reading of the manuscript.


\begin{thebibliography}{99}
\bibitem{AHS} M.F. ATIYAH, \emph{K-Theory}  (Benjamin, New York 1967).
\bibitem{ABP1} A.-M. AUBERT, P. BAUM and R.J. PLYMEN, `Geometric structure in the representation theory of
$p$-adic groups',  \emph{C.R. Acad. Sci. Paris, Ser. I}  345 (2007)
573--578.
\bibitem{ABP2} A.-M. AUBERT, P. BAUM and R.J. PLYMEN, `Geometric structure in the principal series of the $p$-adic group $G_2$',  \emph{Represent. Theory} 
15 (2011)  126 --169. 
\bibitem{ABP3} A.-M. AUBERT, P. BAUM and R.J. PLYMEN, `Geometric structure in the representation theory of $p$-adic groups II', \emph{Contemp. Math.}
534 (2011) 71 --90.
\bibitem{ABP4} A.-M. AUBERT, P. BAUM and R.J. PLYMEN, `Extended quotients in the principal series of reductive $p$-adic groups'. \url{http://arxiv.org/abs/1110.6596}
\bibitem{AHP} A.-M.  AUBERT, S. HASAN and  R.J. PLYMEN, `Cycles in the chamber homology of $\GL(3)$', \emph{K-Theory} 37 (2006) 341 -- 377.
\bibitem{BC} P. BAUM and A. CONNES, `The Chern character for discrete groups', A F\^ete of Topology, Academic Press, New York, 1988, 163 -- 232. 
\bibitem{BHP} P. BAUM, N. HIGSON and R.J. PLYMEN, `Equivariant homology for $\SL(2)$ of a $p$-adic field',  \emph{Contemp. Math.} 148 (1993) 1-18.
\bibitem{Black} B. BLACKADAR, \emph{$K$-Theory for operator algebras}  (Cambridge University Press, 2002.)
\bibitem{B} N. BOURBAKI, \emph{Lie groups and Lie algebras} (chapters 4-6, Springer 2002;
chapters 7-9, Springer 2005).
\bibitem{G} D. GOLDBERG, `$R$-groups and elliptic representations for $\SL_n$', \emph{Pacific J. Math.} 165 (1994) 77 --92.
\bibitem{H} N. HIGSON and V. NISTOR, `Cyclic homology of totally disconnected groups acting on buildings',  \emph{J. Funct. Anal.} 141 (1996) 466 --495.
\bibitem{K} C. D. KEYS, `Reducibility of unramified unitary principal series of p-adic groups and class-1 representations', \emph{Math. Ann.} 260 (1982) 397-402.
\bibitem{VL} V. LAFFORGUE,  `KK-theory for bivariant Banach algebras and the Baum-Connes conjecture', \emph{Invent. Math.} 149 (2002) 1 -- 95.
\bibitem{PL} R.J. PLYMEN and C.W. LEUNG, `Arithmetic aspect of operator algebras', \emph{Compos. Math.} 77 (1991) 293 -- 311.
\bibitem{P} R. J. PLYMEN,  `Reduced $C^*$-algebra for reductive $p$-adic groups', \emph{J. Funct. Anal.} 88 (1990) 251 -- 266.
\bibitem{PS} P. SCHNEIDER, `Equivariant homology for totally disconnected groups', \emph{J. Algebra} 203 (1998) 50 -- 68.
\bibitem{S2} M. SOLLEVELD, `On the classification of irreducible
representations of affine Hecke algebras with unequal parameters', \emph{Represent. Theory} 16 (2012) 1 -- 87.
\end{thebibliography}
\end{document}